\documentclass[12pt]{amsart} %
\usepackage{fullpage}
\usepackage{graphicx}
\usepackage{graphics}
\usepackage{amsmath}
\usepackage{amsfonts}
\usepackage{amssymb}
\usepackage{amsthm}
\usepackage{psfrag}
\usepackage{amsmath,amssymb,color}
\usepackage{amsfonts,amstext}
\usepackage{enumerate}

\newtheorem{theorem}{Theorem}

\newtheorem{conjecture}{Conjecture}
\newtheorem{lemma}{Lemma}

\newtheorem{definition}{Definition}

\newcommand{\red}[1]{\textcolor{red}{#1}}

\title{Extremal permutations in routing cycles}
\date{\today}

\author{Junhua He$^{\dagger}$\and Louis A. Valentin$^{\ddagger}$  \and Xiaoyan Yin$^{\S}$ \and Gexin Yu$^{\star}$ }

\thanks{$^{\dagger}$ School of Mathematical Sciences, University of Electronic Science and Technology of China, Chengdu, Sichuan, 611731, China. Research support in part by the fundamental research funds for Chinese universities under Grant ZYGX2012J120.}

\thanks{$^{\ddagger}$ Department of Mathematics, College of William and Mary, Williamsburg, VA 23185. }

\thanks{$^{\S}$ School of Mathematics and Statistics, Xidian University, Xi'an, Shaanxi,  710071, China.  Research supported in part by National Science Foundation of China (61373174) and the Fundamental Research Funds for the Central Universities (K5051370007).}

\thanks{$^{\star}$ Corresponding author. Email:  gyu@wm.edu.   Department of Mathematics, College of William and Mary, Williamsburg, VA 23185. The research of the last author is supported in part by the NSA grant  H98230-16-1-0316.}

\begin{document}
\maketitle

\begin{abstract}
Let $G$ be a graph whose vertices are labeled $1,\ldots,n$, and $\pi$ be a permutation on $[n]:=\{1,2,\cdots, n\}$.  A pebble $p_i$ that is initially placed at the vertex $i$ has destination $\pi(i)$ for each $i\in [n]$. At each step,  we choose a matching and swap the two pebbles on each of the edges.  Let $rt(G, \pi)$, the routing number for $\pi$, be the minimum number of steps necessary for the pebbles to reach their destinations.

Li, Lu, and Yang proved that $rt(C_n, \pi)\le n-1$ for every permutation $\pi$ on the $n$-cycle $C_n$ and conjectured that for $n\geq 5$, if $rt(C_n, \pi) = n-1$, then $\pi = 23\cdots n1$ or its inverse. By a computer search, they showed that the conjecture holds for $n<8$. We prove in this paper that the conjecture holds for all even $n\ge 6$.
\end{abstract}

\section{Introduction}
Routing problems occur in many areas of computer science. Sorting a list involves routing each element to the proper location. Communication across a network involves routing messages through appropriate intermediaries. Message passing between multiprocessors requires the routing of signals to correct processors.

In each case, one would like the routing to be done as quickly as possible. Let us consider the routing model introduced by Alon, Chung, and Graham \cite{ACG94} in 1994. Let $G = (V,E)$ be a graph whose vertices are labeled as $1, \ldots, n$.  For a permutation $\pi$ on $[n]$, a pebble $p_i$, which has destination $\pi(i)$, is placed at $i$ for each $i\in [n]$.  For example, let $\pi=342165=\begin{pmatrix}1&2&3&4&5&6\\3&4&2&1&6&5\end{pmatrix}$, then the destination of $(p_1, p_2, p_3, p_4, p_5, p_6)$ are $(3,4,2,1,6,5)$.

We wish to move pebbles to their destinations.  To do so, at each round we select a matching of $G$ and swap the pebbles at the endpoints of each edge, and repeat rounds until all pebbles are in place.

Let $rt(G, \pi)$ denote the minimum number of rounds necessary to route $\pi$ on $G$. Then, the {\bf routing number} of $G$ is defined as:
\[
rt(G) = \max_{\pi}\  \{ rt(G, \pi) \}
\]

Note that when each matching in the routing process consists of one edge, the routing number corresponds to the diameter of Cayley graph generated by (cyclically) adjacent transpositions.  The diameter problem of Cayley graphs has a rich research literature, and we refer the reader to the book~\cite{FLRTV09} for a comprehensive survey.

%As the routing problem occurs in many contexts in computer science, some of the first bounds are consequences of computer science algorithms.
%The odd-even transposition sort \cite{K98} and  Benes network \cite{B65} show  $rt(P_n) = n$ and $rt(Q_n) \le 2n- 1$, for $n$-vertex path $P_n$ and $n$-dimensional hypercube $Q_n$, respectively.

Very few results are known for the exact values of the routing numbers of graphs.
\begin{theorem}[Alon, Chung, and Graham~\cite{ACG94}] The following are true:
\begin{enumerate}
\item $rt(P_n)=n$, $rt(K_n) = 2$ and  $rt(K_{n,n}) = 4$;
\item  $rt(G) \ge diam(G)$ and $rt(G) \ge \frac{2}{|C|} \min \{|A|,|B|\}$, where $diam(G)$ is the diameter of $G$ and $C$ is a set that cuts $G$ into parts $A$ and $B$;
\item  $rt(G) \le rt(H)$ and $rt(T_n) < 3n$, where $H$ is a spanning subgraph of $G$ and $T_n$ is a tree on $n$ vertices;
\item $rt(G_1 \times G_2) \le 2rt(G_1)+rt(G_2)$, and $n \le  rt(Q_n) \le 2n-1$.
\end{enumerate}
\end{theorem}

Zhang~\cite{Zheng99} improved the above bound on trees, showing that $rt(T_n) \le \left \lfloor \frac{3n}{2} \right \rfloor + O(\log n)$.

 Li, Lu, and Yang~\cite{LLY10} showed that $n+1 \le rt(Q_n) \le 2n-2$, improving both the previous upper and lower bounds on hypercubes. Among other results, they showed that $rt(C_n)= n-1$, and made the following conjecture.

\begin{conjecture}[Li, Lu, Yang~\cite{LLY10}] \label{conj}
For $n \ge 5$, if $rt(C_n, \pi) = n-1$, then $\pi$ is the rotation $23\cdots n1$ or its inverse $n12\ldots (n-1)$.
\end{conjecture}

The conjecture does not hold for $n=4$; the permutation that transposes two non-adjacent vertices and fixes the other two serves as a counterexample. They verified the conjecture for $n<8$ through a computer search.  The conjecture is kind of counter-intuitive:  the permutations on the cycle requiring most time to route are the ones that each pebble is very close to (actually at distance one from) its destination.

In this article, we give a proof of the conjecture when $n$ is even.

\begin{theorem} \label{thework}
For even $n \ge 6$, if $rt(C_n, \pi) = n-1$, then $\pi$ is the rotation $23\cdots n1$ or its inverse.
\end{theorem}

%this is a new paragraph

Some new tools are introduced in the proof, beyond the ideas from \cite{ALSY11} by Albert, Li, Strang, and the last author.  These tools, introduced in Section~\ref{tools},  enable us to describe precisely the swapping process of each pebble, and determine the pebbles that need $n-1$ steps to route, see Section~\ref{extreme} for detail. In Section~\ref{solution}, we show that the only permutations that need $n-1$ steps to route must be the two permutations in the theorem.

\section{Important notion and tools}\label{tools}

Let $G = C_n$ with even $n\ge 6$ and label the vertices of $C_n$ as $1, 2, \dots , n$ in the clockwise order.   Let the clockwise direction be  the positive direction and counter clockwise be the negative direction.  In the rest of the paper, when list pebbles or sets of pebbles in a row, we always think them to be in the clockwise order on the cycle.

\subsection{The Odd-Even Routing Algorithm}
%The result on the routing number of $P_n$ were shown using what is known as the {\bf odd-even routing algorithm}.

An {\em odd-even sort} or {\em odd-even transposition sort} is a classic sorting algorithm (see \cite{K98}) used to sort a list of numbers on parallel processors. To describe this algorithm,  we may place the $n$ numbers to be sorted on the vertices of the path $P_n=v_1v_2\ldots v_n$. An edge $e = v_iv_{i+1}$ is  odd if and only if $i$ is odd.   At each odd step (respectively, even step) of the routing process, we select a matching consisting of odd edges (respectively, even edges) whose two numbers have the wrong order and swap the numbers on the endpoints.

We apply a similar algorithm on even cycles, whose edges can be partitioned into two perfect matchings. We shall call edges in one perfect matching to be even and the others to be odd.  Thus, once the parity of one edge is specified, the parity of all the edges is determined.    During the odd steps (respectively, even steps) we choose a matching consisting of odd edges (respectively, even edges) whose two pebbles are comparable (defined in subsequent subsections).  If an edge $e$ is chosen to be an odd edge, we would call this algorithm the {\em odd-even routing algorithm with odd edge $e$}.

%Note that this algorithm is not defined on cycles of odd length since the edges that would be labeled as odd edges do not form a matching.

\subsection{Spins and Disbursements}

%%%  spin
There are exactly two paths for  pebble $p_i$ to reach its destination, by traveling either in the positive or negative direction. Let $d^+(i, j)$ denote the distance from the vertex $i$ to the vertex $j$ along the positive direction.  Then if $i<j$, then $d^+(i,j)=j-i$ if $i<j$, and $d^+(i,j)=j-i+n$ if $i>j$.  For simplicity, for pebbles $p_i$ and $p_j$, we define $d^+(p_i,p_j)=d^+(i, j)$, and if $p_i$ and $p_j$ are on the endpoints of an edge, we sometimes call the edge $p_ip_j$.

%%  disbursement
Consider a routing process of a permutation $\pi$ on $C_n$ with pebble set $P=\{p_1, \ldots, p_n\}$. For  each pebble $p_i$, let $s(p_i)$, the {\em spin} of $p_i$, represent the displacement for $p_i$ to reach its destination from its current position. So, $s(p_i) \in \{ d^+(i, \pi(i)), d^+(i, \pi(i))-n \}$. Note that the spin of a pebble changes with its movement.

A sequence $B =(s(p_1), s(p_2), \ldots, s(p_n))$ is called a {\em valid disbursement} of $\pi$ if the spins can be realized by a routing process on $\pi$.    Not all combinations of spins give valid disbursements. The following lemma gives a necessary and sufficient condition for a set of spins to be a valid disbursement.

\begin{lemma}
Let $B=(s(p_1), s(p_2), \ldots, s(p_n))$ be an assignment of the spins to the pebbles.  It is a valid disbursement if and only if $\sum_{i=1}^n s(p_i)= 0$.
\end{lemma}

\begin{proof}
To see the necessity,  we observe that when two pebbles are swapped, one moves forward one step and one moves backward one step, so the sum of spins remains invariant.  Since $B$ is a valid disbursement, the final spins are all zeroes, so the sum is also zero.

For sufficiency, we can move the pebbles one by one along their assigned directions.
\end{proof}

From this lemma, a valid disbursement of a non-identity permutation $\pi$ must contain both positive and negative spins.  Let $s(p_i)>0$ and $s(p_j)<0$ in a valid disbursement $B$ of $\pi$. By {\em flipping the spins of $p_i$ and $p_j$},  we change the spins of  $p_i$ and $p_j$ to $s(p_i)-n$ and $s(p_j)+n$, respectively.   Clearly, after one flip, we obtain a new valid disbursement.

A valid disbursement $(s(p_1), \ldots, s(p_n))$ is minimized if $\sum_{p\in P} |s(p)|$ is minimized.  The following simple fact is very important.

\begin{lemma}
If a valid disbursement is minimized, then $s(p_i)-s(p_j)\le n$ for all $i,j\in [n]$.
\end{lemma}

\begin{proof}
For otherwise, one can make the sum smaller by flipping the spins of $p_i$ and $p_j$.
\end{proof}

%It is not hard to see that the converse is also true, so one can apply the flips on a valid disbursement to get a minimized disbursement.

\iffalse

 The following lemma give a characterization  of minimized disbursement of a permutation.
\begin{lemma}\label{mini-disbursement}
 Let $B=(b_1,\cdots, b_n)$ be a valid disbursement of a permutation $\pi$, the following conditions are equivalent:\\
(1) $B$ is a minimized disbursement of $\pi$;    \\
(2) ${\rm max}(b_1,\cdots, b_n)-{\rm min}(b_1,\cdots, b_n)\le n$.
\end{lemma}

\fi

%Given $C_n$ and a particular valid $B$, the sum of the absolute values of the spins in $B$ is the number of rounds to route  $C_n$ using $B$. %We also have as a lower bound  $\frac{ |B| }{ 2\lfloor \frac{n}{2} \rfloor }$, since the denominator represents the maximum total distance change during one step using a maximum matching. Because of these bounds, we want to minimize $|B|$.

\subsection{An order relation}

Once a valid disbursement $B=(s(p_1), s(p_2), \ldots, s(p_n))$ of $\pi$ is given,  some restrictions are placed on the routing processes realizing $B$.  For example, if $s(p_i)-s(p_j)>d^+(p_i,p_j)$, then $p_i$ and $p_j$ must swap at some round in the routing processes. In other words, each valid disbursement is associated with an order relation on the pebbles.

\begin{definition}
Let $B=(s(p_1), s(p_2), \ldots, s(p_n))$ be a valid disbursement.  We call $p_i\succ p_j$ if $s(p_i)-s(p_j) > d^+(p_i, p_j) $
\end{definition}

%Remark: even though the order is associated with a disbursement, but we will usually omit it in the routing process, as the spins changes and the order is always associated with the new disbursements.

Note that the order relation is transitive. To see that,  let $p_i\succ p_j$ and $p_j\succ p_k$. Then $s(p_i) - s(p_j)>d^+( p_i, p_j)$ and $s(p_j) - s(p_k)>d^+(p_j, p_k)$. It follows that $s(p_i)-s(p_k)>d^+(p_i, p_j)+d^+(p_j, p_k)\ge d^+(p_i, p_k)$, which implies $p_i\succ p_k$.

As two pebbles have different destinations, $s(p_i)-s(p_j)\not=d^+(p_i, p_j)$, so if $p_i\succ p_j$ is not true, then $s(p_i)-s(p_j)<d^+(p_i, p_j)$.  When $p_i\succ p_j$, we say that $p_i$ and $p_j$ are {\em comparable}, or more precisely, {\em $p_i$ is bigger than $p_j$} and {\em $p_j$ is smaller than $p_i$}. If $p_i$ is neither bigger nor smaller than $p_j$, we call them {\em incomparable}.   If each pebble in set $P_1$ is bigger than every pebble in $P_2$, we also write $P_1\succ P_2$.

The following lemma provides a convenient way to determine order relations.

\begin{lemma}\label{induced-order}
Let $x, y, z$ be three pebbles in the clockwise order sitting on the cycle.  If $x\succ z$, then $x\succ y$ or $y\succ z$.  Furthermore, if $x\succ z$, then $y$ is not smaller than $z$ and not bigger than $x$.
\end{lemma}

\begin{proof}
For otherwise, $s(x)-s(y)<d^+(x,y)$ and $s(y)-s(z)<d^+(y,z)$. It follows that $s(x)-s(z)<d^+(x,y)+d^+(y,z)=d^+(x,z)$.  Then $x$ is not bigger than $z$, a contradiction.

For the furthermore part, if $x\succ z$ and $z\succ y$, then $s(x)-s(z)\ge d^+(x,z)$ and $s(z)-s(y)\ge d^+(z,y)$, and it follows that $s(x)-s(y)\ge d^+(x,z)+d^+(z,y)>n$, a contradiction. Likewise, if $x\succ z$ and $y\succ x$, then $s(x)-s(z)\ge d^+(x,z)$ and $s(y)-s(x)\ge d^+(y,x)$, and it follows that $s(y)-s(z)\ge d^+(x,z)+d^+(y,x)>n$, a contradiction.
\end{proof}

The following lemma says that we only need to swap comparable pebbles to route the permutation.

\begin{lemma}\label{incomparable}
Let $B$ be a minimized disbursement of $\pi$. If a pebble $p$ is incomparable with all other pebbles, then $s(p)=0$, i.e., the pebble $p$ is at its destination vertex.
\end{lemma}

\begin{proof}
Suppose that $s(p)\neq 0$.  By symmetry, let $s(p)>0$.

Let $\pi=\Pi_i \pi_i$ be a cycle decomposition of $\pi$, where $\pi_i=(i_1,\cdots ,i_{r_i})$.  Then the pebble placed at $i_{k}$, which we call pebble $i_k$ to save symbols, has  destination $i_{k+1}$ for all $k\le r_i$, with $i_{r_i+1}=i_1$. Let $P_i$ be the set of pebbles on $\pi_i$.  We say that $\pi_i$ is the orbit of the pebbles in $P_i$.  We may assume that $p=i_1$.

We claim that for each $j$,  $\sum_{q\in P_j}s(q)=an$ for some integer $a$.   To see this, we note that $s(j_k)\in \{d^+(j_k, j_{k+1}),  d^+(j_k, j_{k+1})-n\}$.  Thus,  if the spins are all positive, the sum equals $bn$ for some positive integer $b$.  However, each switch of a spin from positive to negative would cause a change of $-n$ in the sum.  So the sum of spins remains a multiple of $n$.

%We also claim that for orbit $P_i$ on which where $p$ is, $\sum_{q\in P_i}s(q)=n$. For convenience, let $p_j$ be on $v_{i_j}$.  We just need to show that all pebbles on $\pi_i$ has positive spins.  For otherwise, if $s(p_k)<0$, then $s(p_k)=d^+(v_{i_k}, v_{i_{k+1}})-n$, so $s(p_1)-s(p_k)=n+d^+(v_{i_1}, v_{i_{2}})-d^+(v_{i_k}, v_{i_{k+1}})>d^+(p_1, p_k)$, and thus $p_1\succ p_k$, a contradiction.

We further claim that all pebbles in $P_i$ have positive spins, which also implies that $\sum_{q\in P_i} s(q)=bn$ for some positive integer $b$.  Note that $s(i_1)=s(p)>0$.  Let $j$ be the smallest integer so that $s(i_j)<0$.  Then $s(i_j)=d^+(i_j, i_{j+1})-n\le -d^+(i_1,i_j)$.  So $s(i_1)>0\ge s(i_j)+d^+(i_1,i_j)$. It follows that $s(i_1)-s(i_j)>d^+(i_1,i_j)$, that is, $p=i_1\succ i_j$, a contradiction.

%We also claim that no pebble in $P_i$ will pass $i_2$ in the negative direction in order to arrive its destination. Otherwise, assume that $i_k (2<k\le r_i)$ is such a pebble, then we have $s(i_k)<-d^+(i_1, i_k)$, or $s(i_k)+d^+(i_1, i_k)<0$.  So  $s(i_1)>0>s(i_k)+d^+(i_1, i_k)$.    It follows that $p=i_1\succ i_k$, a contradiction.

%Furthermore, we have $s(i_2)>0$. Otherwise, let $s(i_2)\le 0$.  As the destination of pebble $i_2$ is $i_3$, we have $s(i_2)=d^+(i_2, i_3)-n<-d^+(i_1, i_2)$.  Hence   $s(i_1)>0>s(i_2)+d^+(i_1, i_2)$. It follows that $p=i_1\succ i_2$, a contradiction.

%The fact that $s(i_1)>0$ implies that $i_1$ will travel in the positive direction, and the fact that $s(i_2)>0$ implies that $i_2$ will arrive its destination $i_3$ in the positive direction, since there exists no pebble will pass $i_2$ in the negative direction.  Therefore, $\sum_{q\in P_1}s(q)=bn$ for some positive integer $b$.

As the sum of all spins is zero, there exists some orbit $\pi_j$ with spin sum $cn$ for some integer $c<0$.  In particular, there exists a pebble $q\in P_j$ such that $q$ passes $p=i_1$ in the negative direction to arrive its destination.  So $s(q)+d^+(p, q)<0<s(p)$ and it follows that $p\succ q$, a contradiction.
\end{proof}

\iffalse
%A routing process with a minimized disbursement has the following nice properties.

\begin{lemma}
If $\sum_{p\in P} |s(p)|$ is minimized for some disbursement $B$, then \\
(a) two pebbles can swap at most once in the routing process;\\
(b) the order on the pebbles is transitive.
\end{lemma}

\begin{proof}
(a) If two pebbles $p_i$ and $p_j$ swap twice in opposite directions, neither swap is necessary.  (Note that this does not change $|B|$.) If they swap twice in the same direction, then $|s(p_i)| + |s(p_j)| > n$. Their spins can be flipped to decrease the absolute sum of $B$.

(b) Suppose $p_i\succ p_j$ and $p_j\succ p_k$. Then, by definition, we have  $s(p_i) - s(p_j)>d^+( p_i, p_j)$ and $s(p_j) - s(p_k)>d^+(p_j, p_k)$.  It follows that $s(p_i)-s(p_k)>d^+(p_i, p_j)+d^+(p_j, p_k)\ge d^+(p_i, p_k)$, so $p_i\succ p_k$.
% Now, either $s(p_k)-s(p_i) = d^+( p_i, p_j) + d^+( p_j, p_k) = d^+( p_i, p_k)$ or $s(p_k)-s(p_i) =n + d^+( p_i, p_k)$. Suppose $s(p_k)-s(p_i) > n$. Then, $s(p_k) > 0$ and $ s(p_i) < 0$, since $ -n < s(p_a) < n$ for all $a$. Flipping the spins of $p_i$ and $p_j$ gives a smaller sum of absolute values of spins, a contradiction. Therefore, $s(p_k) - s(p_i) = d^+( p_i, p_k) \Leftrightarrow p_i \prec p_k$. Thus, the order is transitive.
\end{proof}

\fi

%The above lemma implies that a permutation is the identify if there is no order relation among the pebbles.

Note that the order of pebbles is always associated with the current disbursement, which may not be the same as the initial disbursement.   The following lemma says that whether or not two pebbles swap is determined by the initial disbursement. So we will not keep tracking of the spins, but just see whether necessary swaps are performed.

\begin{lemma}
If $p_i$ and $p_j$ are incomparable, then in the sorting process, they will always be incomparable.   If $p_i\succ p_j$, then $p_i$ and $p_j$ are incomparable after the swap of $p_i$ and $p_j$.
\end{lemma}

\begin{proof}
If $p_i$ and $p_j$ are incomparable, then in the sorting process, $(s(p_i)-s(p_j))-d^+(p_i,p_j)$ does not change: if a pebble swaps with both $p_i$ and $p_j$, then it must be bigger than or smaller than both $p_i$ and $p_j$, thus the distance from $p_i$ to $p_j$ does not change and $s(p_i)-s(p_j)$ does not change;  If a pebble swaps only with $p_i$, then $s(p_i)$ increases by one and $d^+(p_i,p_j)$ increases by one; If a pebble swaps only with $p_j$, then $s(p_j)$ increases by one and $d^+(p_i,p_j)$ decreases by one.

Now let $p_i\succ p_j$. We only need to show that they become incomparable right after the swap of $p_i$ and $p_j$, by what we just proved.  Since $p_i\succ p_j$,  $n\ge s(p_i)-s(p_j)\ge d^+(p_i,p_j)+1\ge 2$.  Let $s'(p_i)$ and $s'(p_j)$ respectively be the new spins of $p_i$ and $p_j$ right after the swap of $p_i$ and $p_j$.  Note that $p_i$ and $p_j$ are adjacent to each other only if the pebbles on the segment from $p_i$ to $p_j$ along the positive direction have swapped with $p_i$ (for those smaller than $p_i$) or $p_j$ (for those bigger than $p_j$).  Then $s'(p_i)-s'(p_j)=s(p_i)-s(p_j)-d^+(p_i,p_j)-2$.  Therefore, $s'(p_j)-s'(p_i)<0$, and $p_j$ cannot be bigger than $p_i$.   Also, $s'(p_i)-s'(p_j)\le n-2<n-1$, thus after the swap, $p_i$ cannot be bigger than $p_j$.
\end{proof}

When $B$ is minimized and two pebbles $p_i$ and $p_j$ satisfy $s(p_i)-s(p_j)=n$,  we still get a minimized disbursement  after flipping the spins of $p_i$ and $p_j$.   The following lemma tells us how the order relation changes when we do such a flip.

\begin{lemma}\label{flip-spins}
 Let $B$ be a minimized disbursement of $\pi$, and $s(p_i)-s(p_j)=n$.  Let $B'$ be the disbursement after flipping the spins of $p_i$ and $p_j$.  Let $k\not\in \{i,j\}$. Then
 \begin{enumerate}
\item $p_j\succ p_i$ under $B'$, and the order relation remains unchanged for pebbles other than $p_i$ and $p_j$;
\item $p_k\succ p_i$ under $B'$ if $p_i$ and $p_k$ are incomparable under $B$; and $p_i$ and $p_k$ are incomparable under $B'$ if $p_i\succ p_k$ under $B$;
\item  $p_j\succ p_k$ under $B'$ if $p_j$ and $p_k$ are incomparable under $B$; and  $p_j$ and $p_k$ are incomparable under $B'$ if $p_k\succ p_j$ under $B$.
\end{enumerate}
\end{lemma}

\begin{proof}
Let $s'(p_i)=s(p_i)-n$ and $s'(p_j)=s(p_j)+n$.

(1) Clearly $p_j\succ p_i$ under $B'$, since $s'(p_j)-s'(p_i)=2n-(s(p_i)-s(p_j))=n>d^+(p_j,p_i).$   For pebbles not in $\{p_i,p_j\}$, the spins and distance do not change from $B$ to $B'$, so their order relation does not change as well.

(2) For $k\notin\{i,j\}$, we know that
$$s(p_k)-s'(p_i)-d^+(p_k, p_i)=s(p_k)-(s(p_i)-n)-(n-d^+(p_i,p_k))=s(p_k)-s(p_i)+d^+(p_i,p_k).$$
Therefore,  $s(p_k)-s'(p_i)>d^+(p_k, p_i)$ if and only if $s(p_i)<s(p_k)+d^+(p_i, p_k)$.  Note that $p_k$ cannot be bigger than $p_i$ under $B$, for otherwise, $s(p_k)-s(p_j)>s(p_i)-s(p_j)=n$.
It follows that $p_k\succ p_i$ under $B'$ if and only if $p_k$ and $p_i$ are incomparable under $B$.

By flipping the spins of $p_i$ and $p_j$ in $B'$, we get $B$.  So we have the other part as well.

(3) Since $s'(p_j)-s'(p_i)=n$ under $B'$, and one gets $B$ after flipping the spins of $p_i$ and $p_j$ in $B'$, these two statements follow from (2).
\end{proof}

\subsection{The window of a pebble}

Let $B$ be a minimized disbursement of $\pi$ with associated order $\succ$. For an arbitrary pebble $p_0$, let    $$U=\{p\in P: p\succ p_0\}\;\;{\rm and}\;\;W=\{p\in P: p_0\succ p\}.$$
By Lemma~\ref{incomparable},  the routing process ends when no pebble has a bigger or smaller pebble.  So we have the following equation, which is heavily used to determine the spins of the pebbles in our later proofs. \begin{equation}\label{spin}
s(p_0)=|W|-|U|.
\end{equation}

By Lemma~\ref{induced-order},  there are no $u\in U, w\in W$ such that $u, w, p_0$ or $p_0, u, w$ on $C_n$.   %For otherwise,  $d^+(p_0, w)+d^+(u,p_0)>n$, thus we have  $s(u)>s(p_0)+d^+(u,p_0)>s(w)+d^+(p_0,w)+d^+(u,p_0)>s(w)+n$,  a  contradiction.
So if $U=\{u_1, u_2, \ldots, u_r\}$ and $W=\{w_1, w_2, \ldots, w_t\}$, then we may assume that the pebbles in $U\cup W$ and $p_0$ are ordered as $u_r, \ldots, u_1, p_0, w_1, \ldots, w_t$ on $C_n$.   We denote the set of pebbles incomparable to $p_0$ between $p_0$ and $w_t$ (between $u_r$ and $p_0$ resp.) by $X$ ($Y$ resp.).

A {\em segment} is a sequence of consecutive pebbles. An $U$-segment is a segment whose pebbles are all in $U$, and likewise, we have $W$-segments, $X$-segments and $Y$-segments.  So we can group the pebbles between $u_r$ and $w_t$ along the positive direction as $$win(p_0)=(U_k, Y_k, U_{k-1}, \ldots, U_1, Y_1, p_0, X_1, W_1, \ldots, X_l, W_l),$$ where $X_1, Y_1$ may be empty, and $win(p_0)$ is called the {\em initial window} of $p_0$.   So in the window $win(p_0)$ of the pebble $p_0$, from the leftmost (the $U$-segment) to $p_0$, the segments are alternatively $U$- and $Y$-segments, and from the rightmost (the $W$-segment) to $p_0$, the segments are alternatively $W$- and $X$-segments.  We shall use this notation without further notice.

We denote the set of all other pebbles as $Z$.   So sometimes we write $\pi$ as $$\pi=(Z, U_k, Y_k, U_{k-1}, \ldots, U_1, Y_1, p_0, X_1, W_1, \ldots, X_l, W_l).$$

%It is clear that  $|win(p_0)|=|U|+|Y|+|W|+|X|+1\le n$ and $p_0$ is in place if and only if $|win(p_0)|=1$. %  If $|win(p)|=1$ for all $p\in P$, it implies that each pebble has reached its own destination and thus the routing process is finished.

By transitivity, we have $u_i\succ w_j$  since $u_i\succ p_0\succ w_j$ for all $1\le i\le r$ and $1\le j\le t$, and in particular, $u_r\succ w_t$, hence $n\ge s(u_r)-s(w_t)>d^+(u_r, w_t)$.   By Lemma~\ref{induced-order},
$\text{if $i\ge j$, then $u\succ y$ for all $u\in U_i,\;y\in Y_j$; If $k\ge l$, then $w\prec x$ for all $w\in W_k,\;x\in X_l$. }$

\bigskip

\subsection{Two important lemmas}

A nice property of the odd-even routing algorithm is the following

%%%%%%%%% the following was the old Lemma 8 (now Lemma 7)

\begin{lemma}\label{consecutive-swaps}
Let $p$ be a pebble and $Q$ be a segment of pebbles.   If $p\succ Q$ or $Q\succ p$, then once $p$ starts to swap with a pebble in $Q$ in an odd-even routing algorithm, $p$ will not stop swapping until $p$ swaps with all pebbles in $Q$ (in the following $|Q|-1$ or more steps). % In particular,  once $p_0$ begins to swap with a pebble in $U_i$ (or $W_j$), $p_0$ will swap with $U$-elements (or $W$-elements) in the next $|U_i|-1$ steps (or $|W_j|-1$ steps).
\end{lemma}

\begin{proof}
By symmetry, we let $p\succ Q$. An {\em enlargement of $Q$} is a segment obtained from $Q$ by mixing some pebbles that  are smaller than $p$.  So $p\succ Q'$ if $Q'$ is an enlargement of $Q$.

We first claim that if $x,y\succ Q$, then after swapping with some pebbles in $Q$, $x$ and $y$ can never be on an edge unless they have swapped with all pebbles in $Q$.  Suppose that the first time such an edge occurs at step $s$.  That is, after step $s-1$, we have $q, x, y, q'$, where $x,y\succ \{q,q'\}$.  Now, at step $s-1$, $xq$ and $yq'$ are among the chosen edges. So, after step $s-1$, we should have $q, x, q',y$, a contradiction.

Note that for a pebble $q$ between $p$ and $Q$ along the positive direction, by Lemma~\ref{induced-order}, either $p\succ q$ or $q\succ Q$; in the former case, $q$ either becomes a pebble in an enlargement of $Q$ or $p,q$ swap before $p$ meets a pebble in $Q$.  So we may assume that all pebbles from $p$ to $Q$ along the positive direction  are bigger than $Q$.  Now by the claim, $p$ never meets a pebble incomparable with $p$  before it finishes swapping with $Q$.  So $p$ moves continuously.
\end{proof}

\begin{lemma}[Rotation Lemma] \label{rot}
Let $q$ be an integer with $ -\frac{n}{2} < q \le \frac{n}{2}$,  and $\pi$ be the permutation that satisfies $\pi(a) = a + q \pmod{n}$ for each $a\in [n]$. Then, $rt(C_n, \pi) = n - |q|$.
\end{lemma}

\begin{proof}
By symmetry we only consider the case when $q>0$. For each pebble $p$, the spin of $p$ is either $q$ or $q-n$. Since the sum of spins is zero, there must be exactly $n-q$ pebbles with the positive spin and $q$ pebbles with the negative spin. So, $n-q \le rt(C_n, \pi)$.

Now, we show that $rt(C_n, \pi) \le n-q$.  Let the pebbles be $p_1, p_2, \ldots, p_n$ on $C_n$. We assign spins to the pebbles so that $s(p_{2i-1})=q-n$ for $1\le i\le q$ and each of the remaining $n-q$ pebbles has the spin $q$.  We use an odd-even routing algorithm with odd edge $p_1p_2$.  As $q\le n/2$, no two pebbles with spin $q-n$ are adjacent.  In the routing process, $p_1$ will be paired with $p_2, p_4, \ldots, p_{2q}, p_{2q+1}, \ldots, p_n$ in the first $n-q$ steps, thus reach its destination, and similar things occur for all other pebbles with negative spins.  When all pebbles with negative spins reach their destinations, there is no comparable pebbles, so each pebble will be in place.  Thus, $\pi$ is routed in $n-q$ steps.
\end{proof}

%We will consider segments which can be partitioned into blocks.    %\red{I really like to unify Lemmas~\ref{block-2}, ~\ref{block-2a} and ~\ref{block-3} here.  I don't know a good way to do that yet. }

\section{Extremal Windows}\label{extreme}

Now we count the steps needed for an arbitrary pebble, say $p_0$, to swap with all comparable pebbles.  Let the initial window of $p_0$ be $$win(p_0)=(U_k, Y_k, U_{k-1}, \ldots, U_1, Y_1, p_0, X_1, W_1, \ldots, X_l, W_l).$$  As we only consider how $p_0$ swaps with other pebbles, we ignore the swaps between pebbles inside each of the segments, and regard them to be incomprable for now. % For this purpose,  we may regard $p_0$ and pebbles in $X, Y, Z$ ($Z$ is the set of the pebble outside of $win(p_0)$) as ones, pebbles in $U$ as twos, and pebbles in $W$ as zeros, and a pebble is bigger than another if and only the corresponding number is larger. The goal is to swap $p_0$ with all twos and zeros.

%Now let us consider the routing process for an (arbitrary) pebble, say $p_0$.  As defined, let  $U_1, U_2, \ldots, U_k$ and $W_1, W_2, \ldots, W_l$ be the segments comparable with $p_0$.  In the routing process, those segments may be mixed and $p_0$ may not swap with them in their initial order.  However, by Lemma~\ref{consecutive-swaps},  if $p_0$ starts to swap with a pebble in one segment $Q$, then $p_0$ will swap with pebbles in the following $|Q|$ steps (not necessary with elements in $Q$ though).  However, if $p_0$ starts to swap with $W_i$, and in the following $|W_i|$ steps, $p_0$ will swap with pebbles that are smaller than $p_0$, thus must be in $W$ as well; we would denote those pebbles $W_i'$.  So we know $|W_i'|=|W_i|$ and they both contain pebbles smaller than $p_0$.    Similarly we define $U_i'$ for $1\le i\le k$.  Note that $U'$ and $W'$ are not necessarily segments any more. So $p_0$ will meet $U$-sequences in the order of $U_1', U_2', \ldots, U_k'$, and $W$-sequences in the order of $W_1', W_2', \ldots, W_l'$, but may meet those sequences in a mixed order.  Let $Z_1, Z_2, \ldots, Z_{k+l}$ be the order the sequences meet $p_0$, so $Z_i\in \{U_1', \ldots, U_k', W_1', \ldots, W_l'\}$.

By Lemma~\ref{consecutive-swaps}, when we apply the odd-even routing algorithm, $p_0$ will meet a segment $S\in \{U_1, \ldots, U_k, W_1, \ldots, W_l\}$ and swap with all the pebbles in $S$ in the following $|S|$ steps.  Assume that $p_0$ meets the segments in the order $S_1, S_2, \ldots, S_{k+l}$, where $S_i\in \{U_1, \ldots, U_k, W_1, \ldots, W_l\}$.

For $i=1, 2,...,k+l-1$, let $\omega_i$ be the waiting time between $S_{i-1}$ and $S_i$, that is,  the number of steps that $p_0$ waits between swapping with the last pebble of $S_{i-1}$ and swapping with the first pebble of $S_i$.   Let $\alpha$ be the largest index such that $\omega_\alpha\neq 0$. %First,  we assume that $Z_\alpha$ is the $t$-th $W$-sequence.

By symmetry, we assume that $S_{\alpha}=W_t$.   Because of the parity, a swap of $p_0$ and $W$ (or $p_0$ and $U$) cannot be followed by a swap of $p_0$ and $U$ (or $p_0$ and $W$).  Therefore, as $\omega_{\alpha+1}=\cdots=\omega_{k+l}=0$, $\cup_{i\ge \alpha}S_i=\cup_{i=t}^l W_i$ and $p_0$ will swap with them continuously until it reaches its destination.   Let $w$ be the pebble in $W_t$ next to $X_t$.  Since we ignore the swaps between pebbles in $W$, $w$ only moves in one direction (counter-clockwise). Then the steps for $p_0$ to be in place are the steps for $p_0$ and $w$ to meet plus $|\cup_{i=t}^l W_i|$.

To meet $p_0$, $w$ has to swap with pebbles in $\cup_{j=1}^t X_j$ and $\cup_{i=1}^k U_i$.   We also note that, $w$ is always paired with a comparable pebble starting from the first or the second step, depending on the parity of the first edge with which $w$ is incident.   Thus the total number of steps for $p_0$ to be in place is:
\begin{equation}\label{eq01}
\sum_{j=1}^{t}|X_j|+\sum_{j=t}^{l}|W_j|+\sum_{i=1}^{k}|U_i|+\delta,
\end{equation}
 where $\delta=0$ if $w\in W_t$ is paired with an $X$-pebble in the first step, and $\delta=1$ otherwise.

By symmetry, if $S_{\alpha}=U_t$, then the number of steps for $p_0$ to be in place is
\begin{equation}\label{eq02}
\sum_{j=1}^{t}|Y_j|+\sum_{j=t}^{k}|U_j|+\sum_{i=1}^{l}|W_i|+\delta,  \text{ where $\delta\in\{0,1\}$.}
\end{equation}

\iffalse
\begin{lemma}
Let $p_0$ be a pebble with initial window $$win(p_0)= win(p_0)=(U_k, Y_k, U_{k-1}, \ldots, U_1, Y_1, p_0, X_1, W_1, \ldots, X_l, W_l).$$  Then the number of steps for $p_0$ to be in place is
$$\sum_{j=1}^{t}|X_j|+\sum_{j=t}^{l}|W_j|+\sum_{i=1}^{k}|U_i|+\delta,$$
 where $\delta=0$ if $w\in W_t$ is paired with an $X$-pebble in the first step, otherwise $\delta=1$; or
 $$\sum_{j=1}^t |Y_j| + \sum_{j=t}^{k} |U_j| +\sum_{j=1}^{l} |W_j|+\delta,$$
 where $\delta=0$ if $w\in U_t$ is paired with an $Y$-pebble in the first step, otherwise $\delta=1$.
 \end{lemma}

\fi

Therefore, every permutation that takes $n-1$ steps to route must contain a pebble $p_0$ such that (when $S_{\alpha}=W_t$)
    \begin{equation} \label{eq1}
    \sum_{j=1}^k |Y_j| + \sum_{j=t+1}^{l} |X_j| +\sum_{j=1}^{t-1} |W_j| + |Z|= \delta, \text{ where $\delta\in \{0,1\}$.}
    \end{equation}
or (when $S_{\alpha}=U_t$)
    \begin{equation} \label{eq2}
    \sum_{j=1}^l |X_j| + \sum_{j=t+1}^{k} |Y_j| +\sum_{j=1}^{t-1} |U_j| + |Z| = \delta, \text{ where $\delta\in \{0,1\}$.}
    \end{equation}

Now we are ready to determine the extreme windows that need $n-1$ steps to route.

\begin{lemma}\label{extremal-windows}
Every permutation that takes $n-1$ steps to route must contain a pebble $p_0$ whose window is one of the following
\begin{enumerate}[(i)]
\item $|win(p_0)|=n$ and $win(p_0)=(p_0, X, W)$ (or $win(p_0)=(U, Y, p_0)$).
\item $|win(p_0)|=n-1$ and $win(p_0)=(U, p_0, X, W)$  (or $win(p_0)=(U, Y, p_0, W)$).
\item $|win(p_0)|=n$, and $win(p_0)=(p_0, X_1, W_1, X_2, W_2)$ and ${\rm min}(|W_1|, |X_2|)=1$ (or $win(p_0)=(U_2,Y_2,U_1,Y_1, p_0)$ and ${\rm min}(|Y_2|, |U_1|)=1$).
\end{enumerate}
In other words, if the window of a pebble is not one of the above ones, then in $n-2$ steps, the pebble will be in place.
\end{lemma}
\begin{proof}
By symmetry, we may assume that \eqref{eq1} holds.   As $\delta=0$ or $1$, all the terms in the left-hand side of \eqref{eq1} are zeros or ones.

We first claim that when $|win(p_0)|=n$, $U_k$ or $W_l$ must be empty.  For otherwise, let $u_p\in U_k$ and $w_q\in W_l$ be the furthest $U$-pebble and $W$-pebble to $p_0$, respectively.  As $|win(p_0)|=n$, no pebble is bigger than $u_p$ and no pebble is smaller than $w_q$ by Lemma~\ref{induced-order},   so $s(u_p) \ge 1+|Y|+|W|$ and $s(w_q) \le -(1+|X|+|U|)$, and it follows that $s(u_p)-s(w_q)\ge n+1$, a contradiction.

{\bf Case 1.} $\delta=0$ or $\delta=|Z|=1$.  Then $\sum_{j=1}^k |Y_j|=\sum_{j=t+1}^l |X_j|=\sum_{j=1}^{t-1} |W_j|=0$.   It follows that $Y=\emptyset$,  $t=1=l$. So $win(p_0)=(U, p_0, X, W)$.  By the above claim, when $|win(p_0)|=n$, $U$ or $W$ must be empty, so we have (i) or (ii) in the lemma, where $X$ (or $Y$) could be empty.

%Note that $\sum_{j=1}^k |Y_j|=0$ implies that $Y=\emptyset$, thus there is at most one $U$-set; $\sum_{j=t+1}^l |X_j|=0$ means there are at most $t$ non-empty $X$-sets (the first $t$ sets); $\sum_{j=1}^{t-1} |W_j|=0$ means that $|W|=0$ (if $t\ge 2$) or there are at most one non-empty $W$-set (that is, $W_1$ when $t=1$). But if $t=2$, $|W|=0$ implies all $X$-sets are actually outside of window of $p_0$, so $|X|=0$ as well.    Therefore, $win(p_0)=(U, p_0, X, W)$ so that if $W=\emptyset$ then $X=\emptyset$.\\

{\bf Case 2.} $\delta=1$ and $|Z|=0$. Then $|win(p_0)|=n$, and one of the following holds:
\begin{itemize}
\item $\sum_{j=1}^k |Y_j|=1$, and $\sum_{j=t+1}^l |X_j|=\sum_{j=1}^{t-1} |W_j|=0$.   Then $|Y|=1, t=l=1$.  So there are at most two $U$-sets, $U_1$ and $U_2$, and when there are two, $Y_1=\emptyset$ and $|Y_2|=1$.  Because of the above claim, we have $win(p_0)=(U_2, y, U_1, p_0)$.

\item $\sum_{j=t+1}^l |X_j|=1$ and $\sum_{j=1}^k |Y_j|=\sum_{j=1}^{t-1} |W_j|=0$.  Then $Y=\emptyset$ and $t=|X_2|=1$. Because of the above claim, $win(p_0)=(p_0, X_1, W_1, x_2, W_2)$.

\item $\sum_{j=1}^{t-1} |W_j|=1$ and $\sum_{j=1}^k |Y_j|=\sum_{j=t+1}^l |X_j|=0$.  Then $Y=\emptyset$, $t=2$ and $|W_1|=1$, and $X_i=\emptyset$ for $i\ge 3$.  Because of the above claim,  $win(p_0)=(p_0, X_1, w_1, X_2, W_2)$.
\end{itemize}

So we have the desired extremal windows in the lemma.
\end{proof}

\section{Proof of Theorem~\ref{thework}}\label{solution}

In this section, we show how to deal with the extremal situations in Lemma~\ref{extremal-windows}.

For each of the extreme windows, we will decompose it into blocks of the following kinds.  Let $q_1, q_2, \ldots, q_s$ be a segment of pebbles. It  is called a {\em block with head $q_1$} if $q_1\succ q_i$ for $i\ge 2$ and the other pebbles are incomparable; it is called a {\em block with tail $q_s$} if $q_i\succ q_s$ for $i<s$ and the other pebbles are incomparable; it is called an {\em isolated block} if none of the pebbles is comparable.

We start with a minimized disbursement $B$ of $\pi$, the permutation that cannot be routed in $n-2$ steps.  By Lemma~\ref{extremal-windows}, $\pi$ should contain a pebble with one of the extreme windows.  We shall determine $\pi$ explicitly and alter the disbursement and/or the odd-even routing algorithm to show that it can be routed in $n-2$ steps.

\iffalse
\begin{lemma}\label{spin}
    \red{(1) $s(p)=|W_p|-|U_p|$, where $U_p:=\{q\in P: q\succ p\}$ and $W_p:=\{q\in P: p\succ q\}$;}

    \red{(2) If $p,q\in P$ such that  $p\succ q$ and $d^+(p,q)=n-1$, then $s(p)-s(q)=n$.}

    \red{ (3) If $p, p_0, q\in P$  are pebbles along the positive direction such that $p\succ p_0, p_0\succ q$ and $d^+(p,q)=n-2$, then $s(p)-s(q)=n$.}

    \red{(4) If $p,q\in P$ such that $s(p)-s(q)=n$, then $p$ and $q$ are maximal and minimal pebble of $P$ respectively, in particular, $s(p)=|W_p|$ and $s(q)=|U_p|$.}
\end{lemma}

\fi

\subsection{Extremal window type 1: $win(p_0)=(p_0,X,W)$ and $|win(p_0)|=n$}

\begin{lemma}
If permutation $\pi$ needs $n-1$ steps to route and some pebble $p_0$ in $\pi$ has $win(p_0)=(p_0, X, W)$ and $|win(p_0)|=n$, then $\pi$ is $23\ldots n1$ or its inverse.
\end{lemma}

\begin{proof}
Let $X=x_1x_2\ldots x_a$ and $W=w_1w_2\ldots w_b$.  Consider the spins of $p_0$ and $w_b$. Note that $spin(p_0)=|W|$ and $s(w_b) \le -(1+|X|)$ since no pebble is smaller than $w_b$ (by Lemma~\ref{induced-order}) and $\{p_0\}\cup X\succ w_b$. So $s(p_0)-s(w_b)\ge n$ and it follows that $s(w_b)=-1-|X|$ and $w_b$ is only comparable with $X\cup \{p_0\}$.  Now repeat the argument for $w_{b-1}, \ldots, w_1$ successively,  we have $s(w_i)=-1-|X|$ for $1\le i\le b$.

%Let $X=x_1x_2\ldots x_a$ and $W=w_1w_2\ldots w_b$ in the clockwise order.  Consider the spins of $p_0$ and $w_b$. Note that \red{$p_0\succ w_b$ and $d^+(p_0, w_b)=n-1$, applying lemma \ref{spin}, we have $s(p_0)=|W|$ and $s(w_b) \le -(1+|X|)$, $p_0$ and $w_b$ are maximal and minimal pebble respectively. In particular,  $w_b$ only swaps with the pebbles in $X\cup p_0$.}  \blue{Now no pebble is smaller than $w_{b-1}$ and use the same argument}, we have $s(w_{b-1})=-1-|X|$ and inductively $s(w_i)=-1-|X|$ for $1\le i\le b$.

Consider the spin of $x_1$. It is clear that $s(x_1)\geq |W|$ since $x_1\succ W$ and no pebble is bigger than $x_1$ (by Lemma~\ref{induced-order}). Then $s(x_1)-s(w_b)\ge n$. It follows that $s(x_1) =|W|$ and $x_1\succ W$ is the only order relation involving $x_1$. Inductively we have $s(x_i) =|W|$ for all $x_i\in X$ and $\{p_0\}\cup X\succ W$ is the only order relation in the permutation.

So along the positive direction every pebble is $|W|$ steps away from its destination.   So $\pi$ is a rotation.   By Lemma \ref{rot}, $\pi$ must be $23\cdots n1$ or its inverse.
\end{proof}

\subsection{Extremal window type 2:  $win(p_0)=(U, p_0, X, W)$ and $|win(p_0)|=n-1$.}

\begin{lemma}\label{block-2}
If a permutation $\pi$ contains a pebble $p_0$ such that $win(p_0)=(U, p_0, X, W)$ and $\pi=(\{z\}, U, p_0, X, W)$,where $U, W\not=\emptyset$, then
%\begin{itemize}
$U$ and $W$ are isolated blocks,  and $X$ can be partitioned into $X_1, \ldots, X_r$ such that $X_i$ is either an isolated block or a block with tail $x_i$. Furthermore,  $s(z)=c\le 0$, and if $c<0$, then $X_r$ is an isolated block with $-c$ pebbles, and the order relations are
$$U\cup\{p_0\}\cup X\succ W,\qquad  X_r\succ \{z\},\qquad  X_i-x_i\succ x_i \text{ for some  $1\le i\le r$}.$$
%\item if $U$ is empty and $s(z)=c>0$, then $X$ is an isolated block and $W$ can be partitioned into isolated blocks and blocks with heads so that the block next to $X$, say $W_0$, is isolated with $c$, and the order relation is $\{p_0\}\cup X\succ W$ and $z\succ W_0$.
%\item if $U$ is empty and $s(z)=c<0$, then $W$ is isolated and $X$ can be partitioned into isolated blocks and blocks with tails so that the block next to $W$ is isolated with $|c|$ pebbles and bigger than $z$, and the only other order relation is $\{p_0\}\cup X\succ W$.
%\item if $U$ is empty and $s(z)=0$, then either $X$ can be partitioned into isolated blocks and blocks with tails and $W$ is an isolated block or $W$ can be partitioned into isolated blocks and blocks with heads and $X$ is an isolated block,  and the only other order relation is $\{p_0\}\cup X\succ W$.
%\end{itemize}
\end{lemma}

\begin{proof}
Let $U=u_1u_2\ldots u_p, X=x_1x_2\ldots x_a$ and $W=w_1w_2\ldots w_b$.  Consider the spins of $u_1$ and $w_b$.  As no pebble is bigger than $u_1$ and $u_1\succ \{p_0\}\cup W$, $s(u_1)\ge 1+|W|=1+b$.  Similarly, no pebble is smaller than $w_b$ and $U\cup \{p_0\}\cup X\succ w_b$, so $s(w_b)\le -(1+|U|+|X|)=-(1+a+p)$.  So $s(u_1)-s(w_b)\ge 1+b+1+a+p=n$, and the equality must hold.  So $s(u_1)=1+b, s(w_b)=-(1+a+p)$ and the only order relation involving $u_1$ and $w_b$ are $u_1\succ \{p_0\}\cup W$ and $U\cup \{p_0\}\cup X\succ w_b$.  Inductively we can consider $u_2$ and $w_{b-1}$ and all pebbles in $U$ and $W$ and conclude that $U\cup \{p_0\}\cup X\succ W$ is the only order relation involving $U$ and $W$.

%Let $U=u_1u_2\ldots u_k, X=x_1x_2\ldots x_a$ and $W=w_1w_2\ldots w_b$ along the positive direction. \red{Applying lemma \ref{spin}, we have $s(u_1)=1+|W|=1+b, s(w_b)=-(1+|U|+|X|)=-(1+a+k)$, $u_1$ and $w_b$ are maximal and minimal pebbles respectively,} the only order relation involving $u_1$ and $w_b$ are $u_1\succ \{p_0\}\cup W$ and $U\cup \{p_0\}\cup X\succ w_b$.  \blue{Inductively} we can consider $u_2$ and $w_{k-1}$ and all pebbles in $U$ and $W$ and conclude that $U\cup \{p_0\}\cup X\succ W$ is the only order relation involving $U$ and $W$.

Now consider the spins of pebbles in $X$.  As $s(w_b)=-(1+a+k)$ and $s(x)-s(w_b)\le n$ for each $x\in X$, we have $s(x)\le b+1$.   Note that $z$ cannot be bigger than any pebble in $X$, for otherwise $z\succ W$ and contradict to what we just concluded.  But $z$ may be smaller than some pebbles in $X$, thus $s(z)\le 0$.

Consider $x_1$.  By Lemma~\ref{induced-order}, no pebble is bigger than $x_1$. As $x_1\succ W$, $s(x_1)\ge |W|=b$. So $s(x_1)\in \{b, b+1\}$.     Let $s(x_1)=b+1$. Then the order relations involving $x_1$ are $x_1\succ W\cup\{x_i\}$ for some $2\le i\le a$ or $x_1\succ W\cup\{z\}$; if $x_1\succ z$, then $x_j\succ z$ for $1\le j\le a$ by Lemma~\ref{induced-order} and we inductively conclude $s(x_j)=b+1$, thus $X$ is an isolated block and $X\succ z$;  if $x_1\succ x_i$ for some $2\le i\le a$, then $x_j\succ x_i$ for $1\le j<i$ by Lemma~\ref{induced-order}, and no other pebble in $X$ is smaller than $x_i$, for otherwise it would be smaller than $x_1$ which contradicts what we just concluded. So $x_1x_2\ldots x_i$ is a block with tail $x_i$.  Now we similarly consider $x_{i+1}$ and get a block partition of $X$.  Now let $s(x_1)=b$. Then $x_1\succ W$ is the only order relation involving $x_1$, and we will inductively consider $x_2$ and get a block partition of $X$.
\end{proof}

Now we are ready to show that such permutations can be routed in $n-2$ steps.

\begin{lemma}
If a permutation $\pi$ contain a pebble $p_0$ such that $win(p_0)=(U, p_0, X, W)$ and $\pi=(\{z\}, U, p_0, X, W)$, where $U, X, W\not=\emptyset$,  then $\pi$ can be routed in at most $n-2$ steps.
\end{lemma}

\begin{proof}
First we assume that $X\not=\emptyset$.  Let $\pi=zu_1\ldots u_kp_0x_1\ldots x_aw_1\ldots w_b$, with $u_i\in U, x_i\in X$ and $w_i\in W$.  We use an odd-even routing algorithm so that $x_aw_1$ is an odd edge.  The order relations are shown in Lemma~\ref{block-2}. %the order relations are  $$U\cup\{p_0\}\cup X\succ W,\qquad  X_0\succ \{z\},\qquad  x_i\succ X_i-x_i \text{ for some $1\le i\le r$}.$$

By Lemma~\ref{consecutive-swaps}, $x_a$ swaps with $w_1$ in the first step, thus swaps with $W$ in the following $|W|-1$ steps, so $w_b$ meets (i.e., is paired with a pebble in) $U\cup \{p_0\}\cup X$ after $|W|-1$ steps, then $w_b$ would swap with $U\cup \{p_0\}\cup X$ in the following $|U\cup \{p_0\}\cup X|$ steps, so it takes $|W|-1+|U\cup \{p_0\}\cup X|=n-2$ steps for $w_b$ to be in place.  As a pebble in $U\cup \{p_0\}\cup X$ has to pass $W-w_b$ to meet $w_b$, all pebbles in $W$ would be in place after $n-2$ steps.

For $z$, its window $win(z)=(X_r, W, z)$, so $|win(z)|=n-|U|-1-|X-X_r|\le n-2$, thus, $z$ will be in place after at most $n-2$ steps, by Lemma~\ref{extremal-windows}.   For a $x_j\in X$ that is in a block with tail $x_i$, its window $win(x_j)=(x_j, \{x_{j+1}, \ldots, x_{i-1}\}, x_i, \{x_{i+1}, \ldots, x_a\}, W),$ so $|win(x_j)|\le n-2-|U|\le n-3$, thus $x_j$ will be in place after at most $n-2$ steps, by Lemma~\ref{extremal-windows}.

So after $n-2$ steps, there are no comparable pebbles, as each order relation involves a pebble in $W\cup \{z\}$ or a pebble in a block with a head in $X$.  By Lemma~\ref{incomparable}, $\pi$ is routed in at most $n-2$ steps.
\end{proof}

%Similarly, $w_1$ swaps with a pebble in $X\cup U\cup \{p_0\}$ in the first step, and swaps with them in the following $|X\cup U\cup \{p_0\}|$ steps.  So $w_1$ meets $p_0$ in $|X\cup U|$ steps and in the meantime, $p_0$ has swapped with $U$, in other words, $p_0$ has swapped with $U$ and meets $W$ after  $|X\cup U|$ steps.   It takes $|W|$ steps for $p_0$ to swap $W$, so $p_0$ will be in place after $|X|+|U|+|W|=n-2$ steps.  % If $X=\emptyset$, $u_p$ meets $\{p_0\}\cup W$ in the second step, then swap with them in the following $|W|+1$ steps, so after $|W|+1$ steps, $p_0$ has swapped with $W$ and meets $U$, thus $p_0$ would be in place after $|W|+1+|U|=n-1$ steps.

%As $p_0$ and $W$ are in place after $n-2$ steps, $U$ are in place after $n-2$ steps, as the only order relation on $U$ is $U\succ \{p_0\}\cup W$.

%So now we only need that all other order relations are taken care within $n-2$ steps.   A head $x_i$ in block $X_i\subseteq X$ is paired with its block in the first or the second step depending on the parity;  in either case, $x_i$ would be swapped with its block or $W$ in the next $|X'|+|W|$ steps, so it will be in place after at most $n-2$ steps.   As $x_1$ meets $W$ in the first step, $x_1$ swaps with $W$ in the next $|W|-1$ steps, thus meets $z$ after $|W|$ steps, and after that, $z$ swaps with the isolated block bigger than it, in at most $|s(z)|\le |X|$ steps, so $z$ would be in place after $|W|+|X|<n-2$ steps.

\begin{lemma}
If a permutation $\pi$ contain a pebble $p_0$ such that $win(p_0)=(U, p_0, W)$ and $\pi=(\{z\}, U, p_0, W)$, where $U, W\not=\emptyset$,  then $\pi$ can be routed in at most $n-2$ steps.
\end{lemma}

\begin{proof}
Let $\pi=zu_1\ldots u_kp_0w_1\ldots w_b$.  By Lemma~\ref{block-2}, $s(u)=1+b$ and $s(w)=-1-k$ for $u\in U, w\in W$ and the only order relation is $U\cup \{p_0\}\succ W$.

First we consider the case when $k=1$ or $b=1$. By symmetry, let $k=1$.   As $n\ge 6$, $b\ge 4$.  Flip the spins of $u$ and $w_2$.  By Lemma~\ref{incomparable}, the order relations under the new disbursement are $p_0\succ W-w_2, w_2\succ (W-w_2)\cup \{z,u\}$.  We use an odd-even routing algorithm with odd edge $p_0w_1$.  Then $w_2$ is paired with a smaller pebble in each step, thus will be in place after $n-2$ steps; similarly, $p_0$ is paired with a smaller pebble in each of the first $n-3$ steps, thus will be in place after $n-3$ steps.  Therefore, after $n-2$ steps, there are no comparable pebbles since each order relation involves $p_0$ or $w_2$.  By Lemma~\ref{incomparable}, $\pi$ is routed in $n-2$ steps.

%Similar to above, $u_k$ meets $W-w_b+p_0$ in the second step thus swaps with them all in the following $|W|$ steps, in other words, $p_0$ meets $u_k$ thus $U-u_1$ after $|W|$ steps, by which $p_0$ has swapped with $W-w_b$.  So it takes $|W|+|U|-1=n-3$ steps for $p_0$ to be in place.

Now, let $k, b\ge 2$.   We first flip the spins of $u_1$ and $w_b$. By Lemma \ref{flip-spins}, the order relations under the new disbursement are $$U-u_1\succ \{p_0\}\cup (W-w_b), \quad w_b\succ (W-w_b)\cup \{z, u_1\},\quad  (U-u_1)\cup \{z, w_b\}\succ u_1.$$

The window for $p_0$ is $win(p_0)=(U-u_1, p_0, W-w_b)$ and the window for $z$ is $win(z)=(w_b, z, u_1)$, so $|win(p_0)|=n-3$ and $|win(z)|=3\le n-3$ (as $n\ge 6$), so by Lemma~\ref{extremal-windows},  $p_0$ and $z$ will be in place after $n-2$ steps.    The window for $u_1$ is $win(u_1)=(U-u_1, W-w_b+p_0, \{w_b, z\}, u_1)$. As $b\ge 2$, it is not one of the extreme windows in Lemma~\ref{extremal-windows}, so $u_1$ will be in place in at most $n-2$ steps.

Now we show that all pebbles in $W-w_b$ are in place after $n-2$ steps, which by Lemma~\ref{incomparable} implies that all pebbles are in place since each order relation involves a pebble in $(W-w_b)\cup \{p_0, z,u_1\}$.   Note that for $1\le i\le b-1$,  $win(w_i)=(w_b, \{z, u_1\}, U-u_1+p_0, \{w_1, \ldots, w_{i-1}\}, w_i)$.  Since $k\ge 2$, $win(u_i)$ is not one of the extreme windows in Lemma~\ref{extremal-windows};  thus, $w_i$ will be in place after $n-2$ steps.
\end{proof}

\subsection{Extremal window type 2a: $win(p_0)=(p_0, X, W)$ and $|win(p_0)|=n-1$. }

This is the case of $win(p_0)=(U, p_0, X, W)$ with $U=\emptyset$. In this case, $W$ is not necessarily an isolated block. The  following lemma tells the possible structures in $\pi$.

\begin{lemma}\label{block-2a}
If a permutation $\pi$ has a pebble $p_0$ with $win(p_0)=(p_0,X,W)$ and $\pi=(\{z\}, p_0, X, W)$, then one of the following must be true (note that $X$ could be empty)
\begin{enumerate}
\item if $s(z)=c>0$,  then $X$ is an isolated block and $W$ can be partitioned into isolated blocks and blocks with heads, say $W_1, W_2, \ldots, W_r$, so that $W_1$ is isolated with $c$ pebbles and $z\succ W_1$, and the only other order relation is $\{p_0\}\cup X\succ W$.
\item if $s(z)=c<0$, then $W$ is isolated and $X$ can be partitioned into isolated blocks and blocks with tails, say $X_1, \ldots, X_r$, so that $X_r$ is isolated with $|c|$ pebbles and $X_r\succ z$, and the only other order relation is $\{p_0\}\cup X\succ W$.
\item if $s(z)=0$, then either $X$ can be partitioned into isolated blocks and blocks with tails and $W$ is an isolated block, or $W$ can be partitioned into isolated blocks and blocks with heads and $X$ is an isolated block,  and the only other order relation is $\{p_0\}\cup X\succ W$.
\end{enumerate}
\end{lemma}

The proof of this lemma is very similar to Lemma~\ref{block-2}, and for completeness, we include a proof below.

\begin{proof}
Let $\pi=zp_0x_1x_2\ldots x_aw_1w_2\ldots w_b$, where $x_i\in X$ and $w_i\in W$.  Clearly, $s(p_0)=|W|=b$.  By Lemma~\ref{induced-order}, $X\succ W$.   As $z$ is incomparable with $p_0$, no pebble in $W$ is bigger than $z$.  %Similarly, no pebble in $X$ is smaller than $z$ (but could be bigger than $z$), thus $z$ could only cause pebbles in $X$ move one step along the positive direction.

%Let $\pi=zp_0x_1x_2\ldots x_aw_1w_2\ldots w_b$ along the positive directions, where $x_i\in X$ and $w_i\in W$.  Clearly, \red{$s(p_0)=|W|=b$ by lemma \ref{spin}}. As $z$ is incomparable with $p_0$, no pebble in $W$ is bigger than $z$ (but could be smaller than $z$), thus $z$ could only cause pebbles in $W$ move one step along the negative direction.  Similarly, no pebble in $X$ is smaller than $z$ (but could be bigger than $z$), thus $z$ could only cause pebbles in $X$ move one step along the positive direction.  % By symmetry we will only consider the case $s(z)=c\ge 0$.

Consider $w_b$. Since $\{p_0\}\cup X\succ w_b$, and no pebble is smaller than $w_b$ by Lemma~\ref{induced-order}, $s(w_b)\le -(|X|+1)=-(a+1)$.  On the other hand, $s(p_0)-s(w_b)\le n$, for otherwise the flip of spins of $p_0$ and $w_b$ gives a smaller disbursement, so $s(w_b)\ge -(a+2)$. Therefore $s(w_b)\in \{-(a+1), -(a+2)\}$.  Clearly, if $s(w_b)=-(a+1)$, then $\{p_0\}\cup X\succ w_b$ is the only order relation involving $w_b$, thus $w_b$ is in an isolated block.  We shall move to consider $w_{b-1}$.

Now let $s(w_b)=-(a+2)$. Then exactly one pebble in $\{z\}\cup (W-w_b)$ is bigger than $w_b$.   If $z\succ w_b$,  then by Lemma~\ref{induced-order}, $z\succ w_j$ for each $w_j\in W-w_b$; thus, inductively we conclude that the only order relation involving $W$ is $\{z, p_0\}\cup X\succ W$, and $W$ is an isolated block.

%\red{For $w\in W$, $s(p_0)-s(w)\le n$ implies $s(w)\ge b-n=-(a+2)$.  Note that $p_0\succ w_b$ implies $s(w_b)\le s(p_0)-d^+(p_0, w_b)-1=b-(n-2)-1=-(a+1)$.} Therefore $s(w_b)\in \{-(a+1), -(a+2)\}$, and $s(w_b)=-(a+2)$ if and only if $w_b$ is smaller than $z$ or some pebble in $W$ but not both \red{since $w_b$ is minimal by lemma \ref{spin}.}  If $w_b$ is not smaller than $z$ and any other pebbles in $W$, then we call $w_b$ to be {\em isolated} and in this case $s(w_b)=-(a+1)$.

If $w_i\succ w_b$ for some $i<b$, then  for $i<j<b$, $w_j$ is not comparable with $w_b$ and  $w_i\succ w_j$ by Lemma~\ref{induced-order}.  Note that no pebble $w_l$ with $l<i$ could be bigger than $w_i$, as otherwise it would be bigger than $w_b$ which is impossible.   Now inductively we conclude that $w_{b-1}, \ldots, w_{i+1}$ all have spin $-(a+2)$ and are only smaller than $\{p_0, w_i\}\cup X$.  That is,  $\{w_i, w_{i+1}, \ldots, w_b\}$ is a block with head $w_i$.

We now repeat the above argument for $w_{b-1}$ (if $s(w_b)=-(a+1)$) or $w_{i-1}$ (if $w_b$ is in a block with head $w_i$), and eventually $W$ can be partitioned into isolated blocks and/or blocks with heads.  In particular, if $z\succ w_r$, then $w_r$ must be in an isolated block, and by Lemma~\ref{induced-order}, $z\succ \{w_1, w_2, \ldots, w_r\}$, that is, $w_1,w_2,\ldots, w_r$ are in an isolated block.

Note that $z$ is not bigger than a pebble in $X$.  For otherwise, let $z\succ x$ for some $x\in X$.  Then $z\succ W$.  Thus, $s(z)\ge 1+b$ and $s(w_b)=-(a+2)$, which implies that $s(z)-s(w_b)\ge n+1$. Now a flip of the spins of $z$ and $w_b$ gives a smaller disbursement, a contradiction.

We claim that if $s(w)=-(a+2)$ for some $w\in W$, then the only order relation involving $X$ is $X\succ W$ (thus $X$ is an isolated block).  Consider $x_1$. Since no pebble is bigger than $x_1$, thus $s(x_1)\ge |W|=b$;  Since $s(x_1)-s(w)\le n$, $s(x_1)\le b$. So $s(x_1)=b$ and the only order relation involving $x_1$ is $x_1\succ W$.   Now we consider $x_2, x_3, \ldots, x_a$ successively and similarly to get the conclusion.  This means also that if $s(z)=c\ge 0$, then the isolated block $w_1w_2\ldots w_i$ has $c$ pebbles.

Similar to the analysis of the structure in $W$, when the order relation involving $W$ is $\{p_0\}\cup X\succ W$, we get the partition and structure in $X$.
\end{proof}

\begin{lemma}
If a permutation $\pi$ contains a pebble $p_0$ with $win(p_0)=(p_0, X, W)$ and $\pi=(\{z\}, p_0, X, W)$, then $\pi$ can be routed in at most $n-2$ steps or is $2\ldots n1$ or its inverse.
\end{lemma}

\begin{proof}
We only consider the case $s(z)=c\ge 0$. By Lemma~\ref{block-2a}, we assume that $X=x_1x_2\ldots x_a$ is an isolated block, and $W=w_1\ldots w_b$ has the block decomposition $W_1, \ldots, W_k$ so that $W_0$ is an isolated block with $c$ pebbles and $z\succ W_1$, and for $i>0$, $W_i$ is either an isolated block or a block with a head (say $w_{i'}$). %the block $W_i$, if not isolated, has head $w_{i'}$ and $W_0$ is an isolated block with $c$ pebbles which are all smaller than $z$.

If $c=b$, then the spins of $\{z, p_0\}\cup X$ are all $b$ and the spins of $X$ are all $-(a+2)$, and $\pi$ is a rotation.  By Lemma~\ref{rot}, if $\pi$ needs $n-1$ steps to route, $\pi$ must be one of the two extremal permutations.  So we assume $c<b$.

When $c=s(z)=0$, we may assume that $W$ is not a block with head $w_1$ (that is, $w_1\succ W-w_1$).  Suppose otherwise.  We flip the spins of $p_0$ and $w_b$, by Lemma~\ref{flip-spins}, $w_b\succ \{z, p_0\}\cup (W-w_b)$, and $X\cup \{z,w_b\}\succ p_0$. Now we use the odd-even sorting algorithm so that $w_bz$ is an odd edge.  By Lemma~\ref{consecutive-swaps}, $w_b$ will be in place in $n-2$ steps, $w_1$ swaps from the second step and takes $n-4$ steps to be in place, and $p_0, z$ will be in place in $3$ steps.  So in at most $n-2$ steps all pebbles will be in place.

Now consider the rest of the cases.  We use the odd-even routing algorithm so that $x_aw_1$ ($p_0w_1$ if $X=\emptyset$) is an odd edge.  By Lemma~\ref{consecutive-swaps}, $x_i$ with $1\le i\le a$ meets $W$ after $a-i$ steps and swaps with $W$ in the following $|W|$ steps, so it will be in place after $a-i+|W|=a+b-i=n-2-i\le n-2$ steps; $p_0$ could be regarded as $x_0$, so takes at most $n-2$ steps; $z$ meets $W_1$ after $a+1$ steps and takes $c$ swaps, so will be in place in at most $a+1+c<a+b+1\le n-2$ steps.   The head $w$ in block $W_i$ swaps with $W_i$ at the first step, or the second step (if it is not adjacent to $x$ and is not incident with an odd edge in the first step), or after $a+2$ steps (if $w=w_1$), and in the former two cases it will be in place after at most $1+(|W_i|-1)+a+1=a+|W_i|+1\le a+b=n-2$ steps. Now consider the last case. If $z\succ w_1$, then $w_1$ is in an isolated block in $W$, thus will be in place after $a+2=n-b\le n-2$ steps (as $b>c\ge 1$). If $w_1$ is not smaller than $z$, then $s(z)=0$, and $w_1$ will be in place after at most $(a+1)+1+|W_i|-1<a+1+|W|=n-1$ steps. So after $n-2$ steps, the above pebbles are in place, thus, all pebbles will be in place by Lemma~\ref{incomparable}, since each order relation involves one of the pebbles.
%{\bf Case 2.} $X=\emptyset$.  If $a=|X|=0$, then we use the odd-even routing algorithm so that $p_0w_1$ is an odd edge.  By Lemma~\ref{consecutive-swaps}, $p_0$ meets $W$ in the first step and will swap with $W$ in the following $|W|=n-2$ steps, and $z$ meets $W_0$ in the second step, so will swap with $W_0$ in $c+1<b+1=n-1$ steps. % For a head $w_i$ in block $W_i$, the window $win(w_i)=(p_0, \cup_{j=0}^{i-1} W_j, w_i, W_i-x_i)$, so it is one of the extreme windows in Lemma~\ref{extremal-windows} only if $W_0, \ldots, W_i$ are the blocks in $W$.
% For a head $w$ of the block $W_i$, it will swap with its block either in the first step, or the second step(if on an even edge) or the third step (if on an even edge and meets $p_0$ in the first step);  for the first case, the head swaps with all pebbles in $W_i$ and $p_0$ in $|W_i|-1+1=|W_i|\le n-2$ steps; for the second case, $|W_i|\le n-3$ and the head swaps with all pebbles in $W_i$ and $p_0$ in $(|W_i|-1)+1+1=|W_i|+1\le n-3+1=n-2$ steps;  for the third case, $|W_i|\le n-2$ and $w$ swaps with all pebbles in $W_i$ and $x$ in $|W_i|-1+2=|W_i|+1\le n-1$ steps.
\end{proof}

\subsection{Extremal window type 3: $win(p_0)=\pi=(p_0, X_1, W_1, x, W_2)$ or $win(p_0)=\pi=(p_0, X_1, w, X_2, W_2)$ .}

\begin{lemma}\label{block-3}
If a permutation $\pi$ contains a pebble $p_0$ with $win(p_0)=\pi=(p_0, X_1, W_1, x, W_2)$ or $win(p_0)=\pi=(p_0, X_1, w, X_2, W_2)$, then $X_1$ and $W_2$ are isolated blocks, and
\begin{itemize}
\item if $win(p_0)=(p_0, X_1, W_1, x, W_2)$, then $W_1$ can be partitioned into isolated blocks and blocks with heads. Furthermore,  $c:=s(x)-|W_2|\ge 0$, and if $c>0$, then the block $W_0$ in $W_1$ next to $X_1$ is an isolated block with $c$ elements and are all smaller than $x$; and the only other order relation between segments are $\{p_0\}\cup X_1\succ W_1\cup W_2, x\succ W_2\cup W_0$.
\item if $win(p_0)=(p_0, X_1, w, X_2, W_2)$, then $X_2$ can be partitioned into isolated blocks and blocks with tails. Furthermore,  $c:=s(w)+(1+|X_1|)\le 0$, and if $c<0$, then the block $X_0$ in $X_2$ next to $W_2$ is an isolated block with $-c$ elements and are all bigger than $w$; and the order relation between segments are $\{p_0\}\cup X_1\succ \{w\}\cup W_2$ and $X_0\succ w$.
\end{itemize}
\end{lemma}

\begin{proof}
We only prove the case when $win(p_0)=(p_0, X_1, W_1, x, W_2)$, and the other case is very similar.
Let $W_1=w_1w_2\ldots w_a$ and $W_2=w_{a+1}w_2\ldots w_b$.  Consider the spins of $p_0$ and $w_b$. Since $s(p_0) = b$ and by Lemma~\ref{induced-order}, no pebble is smaller than $w_b$, thus $s(w_b)\le -(2+|X|)=-(n-b)$; furthermore, as $s(p_0)-s(w_b)\le n$, we have $s(w_b)\ge -(n-b)$, so $s(w_b) = - (n-b)$, and the order relation involving $w_b$ is $\{p_0\}\cup X\succ w_b$. Now we can repeat the argument for $w_{b-1}, w_{b-2}, \ldots, w_{a+1}$ successively and get $s(w_i) = -(n-b)$ for all $w_i \in W_2$. Similarly, by comparing the spins of pebbles in $X_1$ to that of $w_b$, we have $s(x_i) = b$ for all $x_i \in X_1$. So $X_1$ and $W_2$ are isolated blocks.  Since $s(p_0)=s(x)$ for each $x\in X_1$, let $X'_1 = X_1 \cup \{p_0\}$.

%Let $W_1=w_1w_2\ldots w_a$ and $W_2=w_{a+1}w_2\ldots w_b$.  Consider the spins of $p_0$ and $w_b$. \red{It is clear that $s(p_0)=|W_1|+|W_2|=b$ by lemma \ref{spin}(1), and we have $s(w_b)=b-n$ by lemma \ref{spin}(2) since $p_0\succ w_b$ and $d^+(p_0, w_b)=n-1$, $p_0$ and $w_b$ are maximal pebble of $P$ respectively, and the only pebbles that $w_b$ swaps with are $p_0$ and all $x$'s. Since $w_{b-1}$ will not swap with $w_b$, we get $s(w_{b-1}) \le b-n$. Note that $s(p_0)-s(w_{b-1})\le n$, thus we have $s(w_{b-1}) = b-n$ and the only pebbles that $w_{b-1}$ swaps with are $p_0$ and all $x$'s.} By induction, we have $s(w_i) = b-n$ for all $w_i \in W_2$. Similarly, by comparing the spins of pebbles in $X_1$ to that of $w_b$, we have $s(x_i) = b$ for all $x_i \in X_1$ \red{by induction.} Since $s(p_0)$ is the same as the pebbles in $X_1$, let $X'_1 = X_1 \cup \{p_0\}$.   We have shown that $X_1$ and $W_2$ are isolated blocks.

%Now the cycle has been divided into four segments, $X'_1, V_1, X_2,$ and $V_2$, where no segment is empty, either $|X_2| = 1$ or $|V_1|=1$, $s(X'_1) = q$, and $s(V_2) = -(n-q)$. From this point we assume $|X_2| = 1$. Let $x_2$ represent this pebble. (The case when $|V_1| =1$ is symmetric by swapping the labels $X_1$ and $V_2$ and swapping the labels $V_1$ and $X_2$.)

Now we consider the spins of pebbles in $W_1$.

%\red{NOTE: All $w_q, w_b$ have been replaced by $w_a$ in the following two paragraphs. }

We note that no pebble in $W_1$ is bigger than $x$, for otherwise $p_0$ would be bigger than $x$ as $p_0\succ W_1$. Consider $w_a$. The pebbles in $X_1'$ are bigger than $w_a$ and by Lemma~\ref{induced-order}, no pebble is smaller than $w_a$, so $s(w_a)\le -|X_1'|=b-n+1$. On the other hand, $s(p_0)-s(w_a)\le n$ and $s(p_0)=b$ implies $s(w_a)\ge b-n$.  That is, $s(w_a)\in \{b-n, b-n+1\}$, and at most one pebble other than those in $X_1'$ is bigger than $w_a$.

If $s(w_a) = b-n+1$, then $\{p_0\}\cup X_1\succ w_a$ is the only order relation involving $w_a$, and we turn to consider $w_{a-1}$.  If $s(w_a)=b-n$, then $w_a$ is smaller than $x$ or some pebble $w_i\in W_1$; in the former case, all  pebbles in $W_1$ are smaller than $x$ and inductively one can show that they are incomparable and thus $W_1$ is an isolated block; in the latter case, $w_i\succ w_j$ for $i+1\le j\le a$ and $w_i$ is the only such pebble other than those in $X_1'$, so $w_iw_{i+1}\ldots w_a$ is a block with head $w_i$.  Inductively one can have a partition of $W_1$ into blocks, as desired.

We observe that if $x\succ w_i\in W_1$, then $w_i$ is in an isolated block and $x\succ \{w_1, w_2, \ldots, w_i\}$ by Lemma~\ref{induced-order}. Let $c:=s(x)-|W_2|$, then $W_0:=(w_1,\cdots, w_c)$ is an isolated block, $x\succ W_2\cup W_0$ is the only order relation involving $x$, as desired.
\end{proof}

\begin{lemma}
If a permutation $\pi$ contains a pebble $p_0$ such that $win(p_0)=\pi=(p_0, X_1, W_1, x, W_2)$ or $win(p_0)=\pi=(p_0, X_1, w, X_2, W_2)$, then $\pi$ can be routed in $n-2$ steps.
\end{lemma}

\begin{proof}
Again we only consider the case $win(p_0)=(p_0, X_1, W_1, x, W_2)$, as the other one is very similar.
Let $win(p_0)=(p_0x_1x_2\ldots x_kw_1\ldots w_axw_{a+1}\ldots w_b)$ with $x_i\in X_1$ and $w_i\in W_1\cup W_2$.  By Lemma~\ref{block-3}, $X_1, W_2$ are isolated blocks, and $W_1=\cup_{i=0}^r W_1^i$, where $x\succ W_1^0$, and $W_1^i$ for $i>0$ is an isolated block or a block with head $w_1^i$.    We flip the spins of $p_0$ and $w_b$, and then use an odd-even routing algorithm with odd edge $x_kw_1$ to route the permutation.

Now by Lemma~\ref{flip-spins} and Lemma~\ref{block-3}, the order relations under the new disbursement are  $$X_1\cup \{w_b\}\succ W_1\cup (W_2-w_b)\cup \{p_0\},\quad  x\succ (W_2-w_b)\cup W_1^0\cup\{p_0\}, \quad  w_1^i\succ W_1^i-w_1^i \text{ for some $i$}.$$

We list the windows for the pebbles in $X_1\cup \{w_b, x, w_1^i: 1\le i\le r\}$:
\begin{align*}
win(x_i)&=(x_i, X_1-\{x_j: 1\le j\le i\}, W_1, x, W_2-w_b, p_0), \\
win(w_1^i)&=(w_b, p_0, X_1, \cup_{j=0}^{i-1}W_1^j, w_1^i, W_1^i-w_1^i),\\
win(x)&=(x, W_2-w_b, w_b, p_0, X_1, W_1^0),\\
win(w_b)&=(w_b, p_0, X_1, W_1, x, W_2-w_b).
\end{align*}

Then $|win(x_i)|=n-(i-1), |win(w_1^i)|\le n-1-|W_2|\le n-2,  |win(x)|=n-|W_1-W_1^0|$ and $|win(w_b)|=n$, and none of the windows is among the extreme windows in Lemma~\ref{extremal-windows}.  Therefore, in at most $n-2$ steps, the pebbles in $X_1\cup \{w_b, x, w_1^i: 1\le i\le r\}$ will be in place.    However, each order relation on $\pi$ involves one of the pebbles, so by Lemma~\ref{incomparable}, after $n-2$ steps, there are no comparable pebbles, that is, $\pi$ is routed.
 \end{proof}

\section*{Acknowledgement}
 The authors are in debt with an anonymous referee for the valuable comments.

%School of Mathematical Sciences, University of Electronic Science and Technology of China, Chengdu, Sichuan, 611731, P R China

%E-mail address: hejunhua@uestc.edu.cn.


\begin{thebibliography}{1}

\bibitem{ALSY11}
Chase Albert, Chi-Kwong Li, Gilbert Strang, and Gexin Yu.
\newblock Permutations as product of parallel transpositions.
\newblock {\em SIAM J. Discrete Math}, 25(3):1412--1417, 2011.

\bibitem{ACG94}
N.~Alon, F.~R.~K. Chung, and R.~L. Graham.
\newblock Routing permutations on graphs via matchings.
\newblock {\em SIAM J. Discrete Math}, 7(3):513--530, August 1994.

\iffalse
\bibitem{B65}
V.E. Benes.
\newblock {\em Mathematical Theory of Connecting Networks}.
\newblock Academic Press, New York, 1965.
\fi

\bibitem{FLRTV09}
Guillaume Fertin, Anthony Labarre, Irena Rusu, Eric Tannier, and Stphane Vialette. {\em Combinatorics of Genome Rearrangements.} The MIT Press, 1st edition, 2009.

\bibitem{K98}
Donald~E. Knuth.
\newblock {\em The Art of Computer Programming, Volume 3: (2nd ed.) Sorting and
  Searching}.
\newblock Addison Wesley Longman Publishing Co., Inc., Redwood City, CA, USA,
  1998.

\bibitem{LLY10}
Wei-Tian Li, Linyuan Lu, and Yiting Yang.
\newblock Routing numbers of cycles, complete bipartite graphs, and hypercubes.
\newblock {\em SIAM J. Discrete Math}, 24(4):1482--1494, 2010.

\bibitem{Zheng99}
L.~Zhang.
\newblock Optimal bounds for matching routing on trees.
\newblock {\em SIAM J. Discrete Math}, 12:64--77, 1999.

\end{thebibliography}
\end{document}